\documentclass[11pt,oneside]{amsart}
\calclayout

\usepackage{amsmath,amssymb,amsfonts,amsthm,mathabx,color,graphicx,enumitem}
\usepackage[all]{xy}

\usepackage[backend=biber,style=alphabetic,maxalphanames=5,url=false,eprint=false]{biblatex}
\usepackage[colorlinks=true]{hyperref}
\usepackage{xurl}
\hypersetup{breaklinks=true}

\title{An extension theorem for quasimorphisms}
\author{Bingxue Tao}
\newcommand{\llangle}{\left\langle\!\left\langle}
\newcommand{\rrangle}{\right\rangle\!\right\rangle}

\newcommand{\N}{\mathbb N}
\newcommand{\R}{\mathbb R}
\newcommand{\Z}{\mathbb Z}
\newcommand{\lab}{\mathbf{Lab}}

\newcommand{\mcg}{\mathrm{MCG}}

\newcommand{\cl}{\mathrm{cl}}
\newcommand{\scl}{\mathrm{scl}}

\newtheorem{thm}{Theorem}[section]
\newtheorem{lem}[thm]{Lemma} 
\newtheorem{cor}[thm]{Corollary}
\newtheorem{prop}[thm]{Proposition}
\newtheorem{que}[thm]{Question}

\theoremstyle{definition}
\newtheorem{defn}[thm]{Definition}

\theoremstyle{remark}

\begin{document}

\begin{abstract}
    We provide a general sufficient condition for extendability of quasimorphisms on subgroups. This condition recovers the result of Hull--Osin on quasimorphisms on hyperbolically embedded subgroups, and the proof given in this paper is much simpler. We also obtain new results for quasimorphisms on normal subgroups. One result is that for a group $G$ and its normal subgroup $K$, if the quotient $G/K$ is hyperbolic, then any antisymmetric quasi-invariant quasimorphism on $K$ extends to $G$. As an application, the stable commutator length $\mathrm{scl}_G$ is bi-Lipschitz equivalent to the stable mixed commutator length $\mathrm{scl}_{G,K}$ on $[G,K]$. 
    
    Another result concerns about group-theoretic Dehn filling in the sense of Dahmani--Guirardel--Osin. As an application, the quotient of a mapping class group of a surface with boundary by the normal closure of a large power of a pseudo-Anosov element is hierarchically hyperbolic. This gives an affirmative answer to a question of Fournier-Facio--Mangioni--Sisto. 
\end{abstract}

\maketitle

\section{Introduction}

    Let $G$ be a group. A function $\varphi:G\to \R$ is called a \emph{quasimorphism} or a \emph{pseudo-character} if its \emph{defect} 
    \[D(\varphi):=\sup_{g,h\in G} |\varphi(g)+\varphi(h)-\varphi(gh)|\] 
    is finite. Quasimorphisms are closely connected to bounded cohomology of groups and have been studied for a long time within geometric group theory and symplectic geometry. There are also substantial connections to low-dimensional topology and group actions on the circle (see \cite{Cal09a,PR14,Fri17,KKM19}).     Many fundamental constructions (e.g. Brooks quasimorphisms on mapping class groups and surface diffeomorphism groups \cite{BF02,BHW22}, Calabi quasimorphisms on Hamiltonian diffeomorphism groups of symplectic manifolds \cite{EP03,Py06}) arise as extensions of quasimorphisms on certain subgroups to the whole group. A natural but delicate problem is: 

    \begin{que}\label{QMProb}
         Given a subgroup $K<G$ and a quasimorphism $\varphi$ on $K$, when and how can $\varphi$ be extended to a quasimorphism on $G$?
    \end{que}

    A remarkable answer to Question \ref{QMProb} was obtained by Hull and Osin.

    \begin{thm}\cite{HO13}\label{HOHypEmb}
        If $K$ is a hyperbolically embedded subgroup of $G$, then any antisymmetric quasimorphism on $K$ extends to $G$.
    \end{thm}

    In fact, Hull and Osin proved the theorem for $1$-quasicocycles with more general coefficients. Later, their result has been generalized to higher-dimensional quasicocyles by Frigerio, Pozzetti, and Sisto \cite{FPS15}. As an application, they prove that the third bounded cohomology of an acylindrically hyperbolic group is infinite-dimensional. 
    
    On the other hand, Kawasaki, Kimura, Maruyama, Matsushita, and Mimura have a series of work studying Question \ref{QMProb} for $K$ being a normal subgroup of $G$. We refer to the survey \cite{KKMMM24} and the more recent paper \cite{KKMMM25a} as an update. 
    As a notable extrinsic application, four of them \cite{KKMM23} showed the vanishing of the cup product of the fluxes of commuting symplectomorphisms for any closed oriented surface of genus at least two by examining the extendability of certain Calabi quasimorphisms. Note that an infinite normal subgroup cannot be hyperbolically embedded \cite[Proposition 4.33]{DGO17}. Therefore, their work addresses entirely different situations from Theorem \ref{HOHypEmb}. As we will see, one advantage of our result is that Theorem \ref{MainThm} applies to both. 
    
    Let $X$ be a (symmetric) relative generating set of $G$ with respect to $K$. Let $\Gamma(G,K\sqcup X)$ be the Cayley graph of $G$ with respect to the generating set $K\sqcup X$. In this paper, we introduce a new concept, which we call $(G,X)$-\emph{controlled} quasimorphisms. 
    Roughly speaking, a quasimorphism $\varphi$ on $K$ is $(G,X)$-controlled if the $\varphi$-value of the product of all $K$-letters appearing in a cycle of length $L$ in $\Gamma(G,K\sqcup X)$ is bounded by a function with respect to $L$. We refer the reader to Definition \ref{Def_Contr} for the precise definition. 
    Let $Q_{as}(K)$ be the set of all antisymmetric quasimorphisms on $K$, and let $Q_{as}(K)^G$ be its subset of all $G$-quasi-invariant ones. Our main theorem is as follows.
    
    \begin{thm}\label{MainThm}
        Assume that $\Gamma(G,K\sqcup X)$ is $\delta$-hyperbolic for some constant $\delta\ge 0$. Let $\varphi\in Q_{as}(K)$. If $\varphi$ is $(G,X)$-controlled, then $\varphi$ extends to a quasimorphism $\Phi\in Q_{as}(G)$. Moreover, if $\varphi$ is linearly $(G,X)$-controlled with a constant $C_0>0$, then $D(\Phi)\le MC_0+MD(\varphi)$, where $M$ is a constant only depending on $\delta$.
    \end{thm}

    By definition of hyperbolically embedded subgroups (Definition \ref{DefHypEmb}) and Proposition \ref{Prop_CTviaRP}, if $K$ is hyperbolically embedded in $G$, then there exists a relative generating set $X$ such that the assumption of Theorem \ref{MainThm} is satisfied. Thus, Theorem \ref{MainThm} implies Theorem \ref{HOHypEmb} directly. Furthermore, the proof we give in this paper is much simpler than that in \cite{HO13}.

    Being $(G,X)$-controlled is an appropriate assumption for extendability, because it is necessary in many situations. 

    \begin{lem}[Lemma \ref{FiniteCase}]
        If $G$ is generated by $K$ and a finite set $X$, then an extendable quasimorphism on $K$ must be linearly $(G,X)$-controlled. 
    \end{lem}

    In \cite[Theorem 1.9]{KKMMM25}, Kawasaki, Kimura, Maruyama, Matsushita, and Mimura give a sufficient condition on extendability of invariant quasimorphisms on a normal subgroup. It is required that the quotient group has trivial bounded cohomology in lower degree. In contrast to their result, Theorem \ref{MainThm} gives the following corollary, which is a special case of Corollary \ref{RHGQuo}. 
    
    \begin{cor}\label{HypQuo}
        Let $G$ be a group and $K$ a normal subgroup. If $G/K$ is hyperbolic, then any $\varphi\in Q_{as}(K)^G$ extends to $G$. 
    \end{cor}

    For group-theoretic Dehn filling, we obtain the following. 
    
    \begin{thm}[Theorem \ref{QMDGOFill}]\label{QMDGOFill_Intro}
        Let $H$ be a hyperbolically embedded subgroup of $G$. There exists a finite subset $F\subset H-\{1\}$ such that the following holds. Let $N\lhd H$ be a subgroup such that $N\cap F=\emptyset$. Let $K$ be the normal closure of $N$ in $G$. Let $\varphi\in Q_{as}(K)^G$. If there exists a transversal $L$ of $N$ in $H$ such that the restricted quasimorphism $\varphi|_N$ is $(H,L)$-controlled, then $\varphi$ extends to $G$. 
    \end{thm}

    By applying Theorem \ref{QMDGOFill_Intro} to mapping class groups, we obtain new examples of hierarchically hyperbolic groups, and thus answer a question asked by Fournier-Facio, Mangioni, and Sisto \cite[Question 4.2]{FMS25}. 

    \begin{cor}[Corollary \ref{QuoMCG}]
        Let $\Sigma$ be a finite-type surface with boundary. Let $g$ be a pseudo-Anosov mapping class of $\Sigma$. There exists $N\ge 1$ such that the following holds. For any $k\in\N$, let $K$ be the normal closure of $g^{kN}$ in $G$. The quotient group $\mcg(\Sigma)/K$ is hierarchically hyperbolic. 
    \end{cor}

    Motivated by \cite{HO13} and \cite{FPS15}, we ask the following question.

    \begin{que}
        Can Theorem \ref{MainThm} be generalized to quasicocycles with more general coefficients or of higher degree in some sense?
    \end{que}

\subsection*{Structure of the paper}
    In Section \ref{Pre}, we recall some background on quasimorphisms and relative presentations. 
    In Section \ref{Sec_CTviaRP}, we introduce the core definition of this paper, i.e., controlled quasimorphisms and provide a partial characterization of them using relative presentations. Then we prove Theorem \ref{MainThm} in Section \ref{ProofMain}. In Section \ref{Sec_Normal}, we focus on normal subgroups. We first give a criterion for being a controlled quasimorphism on a normal subgroup and then prove Theorem \ref{QMDGOFill_Intro} and its applications to mapping class groups. Finally, we introduce an application to (mixed) stable commutator length. 

\subsection*{Acknowledgement}
    The author is grateful to Francesco Fournier-Facio, Koji Fujiwara, Yusen Long, Giorgio Mangioni, Alexis Marchand, Shuhei Maruyama, Alessandro Sisto, and Renxing Wan for helpful discussions. This work was supported by JST SPRING, Grant Number JPMJSP2110. 

\section{Preliminaries}\label{Pre}
\subsection{Quasimorphisms}
    Let $\hat Q(G)$ denote the set of all quasimorphisms on $G$. 
    Let $K<G$ and $\varphi\in\hat Q(K)$. If $\varphi$ extends to a quasimorphism $\Phi\in\hat Q(G)$, we have 
    \[|\Phi(ghg^{-1})-\Phi(h)|\le |\Phi(g)+\Phi(g^{-1})|+2D(\Phi)\le4D(\Phi)\]
    for any $g,h\in G$. Restricting to $K$, we know that 
    \begin{align}\label{Qinv}
        D^G(\varphi):=\sup \{|\varphi(k_1)-\varphi(k_2)|\mid k_1,k_2\in K \text{ that are conjugate in }G\}<\infty.
    \end{align}
    Thus, (\ref{Qinv}) gives a necessary condition for extendability of $\varphi$. 
    We say that a quasimorphism $\varphi\in\hat Q(K)$ is $G$-\emph{quasi-invariant} if it satisfies (\ref{Qinv}). Let
    \[\hat Q(K)^G:=\{\text{all $G$-quasi-invariant quasimorphisms on }G\}.\]
    The above discussion shows that there is a restricting map
    \[\hat r:\hat Q(G)\to \hat Q(K)^G.\]
    Therefore, Question \ref{QMProb} is essentially asking what the image of $\hat r$ is. Note that we do not explicitly require quasi-invariance in Theorem \ref{MainThm}. This is because controlled quasimorphisms are always quasi-invariant (see Lemma \ref{CTandQI}). 

    A quasimorphism $\varphi$ on $G$ is \emph{antisymmetric} if $\varphi(g^{-1})=-\varphi(g)$ for any $g\in G$. We write $Q_{as}(G)$ to mean the set of all antisymmetric quasimorphisms on $G$ and write $Q_{as}(K)^G$ to mean the subset of all $G$-quasi-invariant ones.

\subsection{Relative presentations}
    For any generating set $Z$ of a group $G$, we write $Z^*$ to mean the set of finite words on $Z$. For any $w_1,w_2\in Z^*$, we write $w_1=_G w_2$ to mean that $w_1$ and $w_2$ represents the same element in $G$. 
    
    A \emph{relative presentation} of $G$ with respect to a subgroup $K$ and a relative generating set $X$ is a presentation of the form
    \begin{align}\label{RelPre}
        G=\langle K,X\mid R_K,R\rangle,
    \end{align}
    where $R_K\subset K^*$ is the set of all words on $K$ that represent the identity of $K$. 
    A word on $K\sqcup X$ is \emph{reduced} if it does not contain a subword of the form $k_1k_2$ with $k_1,k_2\in K-\{1\}$ or a subword of the form $xx^{-1}$ with $x\in X$. Throughout this paper, we can always assume that $R$ only contains reduced words. Thus, any word in $R$ can be identified with an element in $K*F(X)$. 
    
    For any $a\in K*F(X)$, we write $|a|$ to mean the word length of $a$ on the generating set $K\sqcup X$. The unique word on $K\sqcup X$ whose length achieves $|a|$ is called the \emph{normal form} of $a$.

    Let $\theta:K*F(X)\to G$ be the natural projection. Clearly, $\ker(\theta)$ is the normal closure of $R$ in $K*F(X)$.
    A function $f:\N\to \R$ is called a \emph{relative isoperimetric function} of (\ref{RelPre}) if for every element $a\in \ker(\theta)$ with $|a|\le n$, there exists $m\le f(n)$ and an expression 
    \[a=\prod_{i=1}^{m}h_ir_ih_i^{-1},\]
    where $h_i\in K*F(X)$, $r_i\in R$. 

    \begin{defn}
        The relative presentation (\ref{RelPre}) is \emph{bounded} if the words from $R$ have uniformly bounded lengths. 
        The relative presentation (\ref{RelPre}) is \emph{strongly bounded} if it is bounded and the set of letters from $K$ appearing in words from $R$ is finite. 
        The relative presentation (\ref{RelPre}) is \emph{finite} if $R$ is finite. 
    \end{defn}

    Note that a finite relative presentation is always strongly bounded. We refer to \cite{DGO17} for more discussions. In particular, there is a characterization for weak hyperbolicity using bounded relative presentations as below. 
    
    \begin{defn}
        We say that $G$ is \emph{weakly hyperbolic} relative to $K$ and $X$ if the Cayley graph $\Gamma(G,K\sqcup X)$ is hyperbolic. 
    \end{defn}

    \begin{lem}\cite[Lemma 4.9]{DGO17}\label{RelPresent}
        Let $K<G$. Let $X$ be a relative generating set of $G$ with respect to $K$. The following are equivalent. 
        \begin{enumerate}
            \item $G$ is weakly hyperbolic relative to $K$ and $X$. 
            \item There exists a bounded relative presentation $G=\langle K,X\mid R_K,R\rangle$ with a linear relative isoperimetric function. 
            \item There exists a bounded presentation $G=\langle K,X\mid R_K',R\rangle$ with a linear isoperimetric function in the classical sense, where $R_K'\subset R_K$ consists of the words of length at most $3$ and thus can be seen as the multiplication table of $K$. 
        \end{enumerate}
    \end{lem}

    For strongly bounded relative presentations, we recall the following. 
    
    \begin{defn}\cite[Theorem 4.24]{DGO17}\label{DefHypEmb}
        A subgroup $H<G$ is \emph{hyperbolically embedded} with respect to a subset $X\subset G$ if there exists a strongly bounded relative presentation $G=\langle H,X\mid R_H,R\rangle$ with a linear relative isoperimetric function. 
    \end{defn}

    We write $H\hookrightarrow_h (G,X)$ to mean that $H$ is a hyperbolically embedded subgroup of $G$ with respect to $X$. We also write $H\hookrightarrow_h G$ to mean that $H\hookrightarrow_h (G,X)$ for some $X\subset G$.

    For finite relative presentations, we recall the following definition of relatively hyperbolic groups given by Osin \cite{Osi06}, which is equivalent to that of Farb's \cite{Far98}. 
    
    \begin{defn}\label{DefRHG}
        We say that $G$ is \emph{hyperbolic} relative to a subgroup $K$ and a relative generating set $X$ if there exists a finite relative presentation $G=\langle K,X\mid R_K,R\rangle$ with a linear relative isoperimetric function. 
    \end{defn}

\section{Controlled quasimorphisms}
\label{Sec_CTviaRP}
    We are going to introduce the core definition of this paper, i.e., controlled quasimorphisms. 
    Throughout this section, let $G$ be a group and $K$ a subgroup. Let $X$ be a (symmetric) relative generating set of $G$ with respect to $K$. Let $F(X)$ be the free group with the basis $X$. There are two natural projections 
    \[K\xleftarrow{\eta}K*F(X)\xrightarrow{\theta}G,\]
    where $\eta$ maps every $x\in X$ to the identity and every $k\in K$ to itself, and $\theta$ is given by the assumption that $K\sqcup X$ generates $G$.

    For any $\varphi\in \hat Q(K)$, let $\tilde\varphi$ be the pull-back of $\varphi$ via $\eta$. Specifically, \[\tilde\varphi(a):=\varphi(\prod_{i=1}^n k_i),\]
    where $k_i$'s are the $K$-letters that appear in the normal form of $a\in K*F(X)$ in order. Clearly, $\tilde\varphi$ is a quasimorphism on $K*F(X)$ with the same defect as $\varphi$.

    \begin{defn}\label{Def_Contr}
        Let $f:\N\to \R$ be a function and $\varphi\in \hat Q(K)$. 
        We say that $\varphi$ is $(G,X)$-\emph{controlled} by $f$ if for any $a\in \ker(\theta)$ with $|a|\le n$, we have $|\tilde\varphi(a)|\le f(n)$. We say that $\varphi$ is \emph{linearly} $(G,X)$-controlled with a constant $C_0$ if $\varphi$ is $(G,X)$-controlled by a linear function $f(n)=C_0n+C_0$.
    \end{defn}

    First, we explain the motivation behind the definition of controlled quasimorphisms. 
    Let
    \begin{align}\label{CentrExt}
        1\to \Z\xrightarrow{\iota} E\xrightarrow{\pi} G\to 1.
    \end{align}
    be a central extension of groups. Such an extension defines a cohomology class $[\omega]\in H^2(G,\Z)$ as follows. A \emph{(set-theoretic) section} is a map $s:G\to E$, not necessarily a homomorphism, such that $\pi\circ s=1_G$. Choose any section $s:G\to E$. 
    Let $\omega_s$ be a $2$-cochain on $G$ with $\Z$-value defined by 
        \[\omega_s(g_1,g_2):=s(g_1)s(g_2)s(g_1g_2)^{-1}.\]
    It is known that $\omega_s$ is in fact a $2$-cocycle, and the corresponding $2$-coclass $[\omega_s]\in H^2(G,\Z)$ does not depend on the choice of $s$. This $2$-coclass is called the \emph{Euler class} of the central extension (\ref{CentrExt}). 
    
    Let $\omega:G^2\to\Z$ be $2$-cochain. We say that $\omega$ is     
    \begin{itemize} 
        \item \emph{bounded} if the image $\omega(G,G)$ is bounded.
        \item \emph{weakly bounded} if $\omega(g,G)$ and $\omega(G,g)$ are both bounded for all $g\in G$. 
    \end{itemize}
    Correspondingly, a $2$-coclass in $H^2(G,\Z)$ is called \emph{bounded} (resp. \emph{weakly bounded}) if it is represented by a bounded (resp. weakly bounded) $2$-cocycle. For more background on bounded and weakly bounded Euler classes, we refer to \cite{FS23}. 
    The following lemma is known (see \cite[Proposition 2.9]{FMS25} or \cite[Lemma 5.6]{TW25}). 

    \begin{lem}\label{BddEu}
        The central extension (\ref{CentrExt}) has a bounded Euler class if and only if there is a quasimorphism $\Phi:E\to \R$ such that $\Phi\circ \iota=1_{\Z}$. 
    \end{lem}
    
    Back to the 1990s, Neumann and Reeves \cite{NR97} proved that any such central extension with $G$ being hyperbolic has a bounded Euler class. A key step in their proof is to employ a maximizing technique similar to \cite{EF97} to prove the following, which is somehow implicit in their paper.

    \begin{lem}\label{WeakBdd}
        If the central subgroup $\Z$ is undistorted, then the Euler class of (\ref{CentrExt}) is weakly bounded. 
    \end{lem}

    Note that Lemma \ref{WeakBdd} gives an answer to \cite[Question 1.9]{FS23}. 
    Definition \ref{Def_Contr} is inspired by the use of undistortedness of $\Z$ in their proof. Indeed, assume that $E$ is finitely generated by $X$. It is easy to see that the identity map $1_{\Z}$ is linearly $(E,X)$-controlled if and only if the central subgroup $\Z$ is undistorted in $E$.

    The following lemma shows that a linear control is necessary for extendability if $X$ is finite. 
    
    \begin{lem}\label{FiniteCase}
        If $G$ is generated by $K$ and a finite set $X$, then an extendable quasimorphism on $K$ must be linearly $(G,X)$-controlled.
    \end{lem}

    \begin{proof}
        Assume that $\varphi\in \hat Q(K)$ extends to $\psi\in \hat Q(G)$. Let $\tilde\psi\in \hat Q(K*F(X))$ be the pull-back of $\psi$ via $\theta$. Let $B:=\max\{|\psi(x)|\mid x\in X\}$. 

        Let $a\in \ker(\theta)$. Let $x_1,\dots,x_m$ and $k_1,\dots,k_n$ be the $X$-letters and the $K$-letters in the normal form of $a$ in order, respectively. Note that $|a|=m+n$. Since $\tilde\psi$ is a quasimorphism and $\psi|_K=\varphi$, we have
        \[|\tilde\psi(a)-\sum_{i=1}^n \varphi(k_i)-\sum_{j=1}^m \psi(x_j)|\le (m+n-1)D(\psi).\]
        By definition, 
        \begin{align*}
            |\tilde\varphi(a)-\sum_{i=1}^n \varphi(k_i)|&=|\varphi(\prod_{i=1}^n k_i)-\sum_{i=1}^n \varphi(k_i)|\\
            &\le (n-1)D(\varphi).
        \end{align*}
        Thus, 
        \begin{align*}
            |\tilde\psi(a)-\tilde\varphi(a)|&\le mB+(m+n-1)D(\psi)+(n-1)D(\varphi)\\
            &\le |a|(B+D(\psi)+D(\varphi)).
        \end{align*}
        The conclusion follows because $|\tilde\psi(a)|=|\varphi(1)|\le D(\varphi)$.
    \end{proof}
    
    When $X$ is allowed to be infinite, it is easy to construct counterexamples of Lemma \ref{FiniteCase}. Let $\Phi\in \hat Q(G)$. Assume that the restriction quasimorphism $\varphi=\Phi|_K$ is unbounded. Take a sequence of elements $\{k_n\}_{n\in\N}$ in $K$ such that $\varphi(k_n)\to \infty$. Fix any $x\in X$. We can produce a new relative generating set $X'$ from $X$ by adding the elements $x_n:=k_n^{-1}x\in G$ for every $n\in \N$. Now any element of the form $k_nx_nx^{-1}$ lies in $\ker(\theta)$ and has word length at most $3$, while $\tilde\varphi(k_nx_nx^{-1})\to \infty$ as $n\to \infty$. However, the following question does not seem obvious.

\begin{que}
    Does there exist a group $G$, a subgroup $K$, and an extendable quasimorphism $\varphi\in \hat Q(K)$ such that $\varphi$ is not linearly $(G,X)$-controlled for every relative generating set $X$?
\end{que}

    The following lemma shows that controlled quasimorphisms are quasi-invariant. We will not use it, but we prove it here for completeness of this paper. 

    \begin{lem}\label{CTandQI}
        For any $\varphi\in \hat Q(K)$, if $\varphi$ is $(G,X)$-controlled, then $\varphi$ is $G$-quasi-invariant with $D^G(\varphi)\le 2D(\varphi)$. 
    \end{lem}

    \begin{proof}
        Let $k\in K$ and $g\in G$ such that $k'=_Ggkg^{-1}\in K$. Choose any $a\in \theta^{-1}(g)$. Consider the element $ak^na^{-1}k'^{-n}\in \ker(\theta)$, where $n\in \Z$. By definition, \begin{align*}
            |\tilde\varphi(ak^na^{-1}k'^{-n})|&=|\varphi(\eta(a)k^n\eta(a)^{-1}k'^{-n})|\\
            &\ge |\varphi(k^n)+\varphi(k'^{-n})|-5D(\varphi)\\
            &\ge n|\varphi(k)-\varphi(k')|-(2n+5)D(\varphi).
        \end{align*} 
        On the other hand, 
        \begin{align*}
            |\tilde\varphi(ak^na^{-1}k'^{-n})|&\le f(|ak^na^{-1}k'^{-n}|)\\
            &\le f(2|a|+2),
        \end{align*} 
        where $f$ is the function that controls $\varphi$. Combining the above inequalities and letting $n$ be arbitrarily large, we can conclude that $|\varphi(k)-\varphi(k')|\le 2D(\varphi)$, and thus finish the proof. 
    \end{proof}

    The following proposition gives a partial characterization of controlled quasimorphisms using relative presentations.

    \begin{prop}\label{Prop_CTviaRP}
        Let 
        \begin{align}\label{RelPre2}
            G=\langle K,X\mid R_K,R\rangle
        \end{align}
        be a relative presentation. Let $\varphi\in \hat Q(K)$.
        If (\ref{RelPre2}) has a relative isoperimetric function $f$ and $\sup\{|\tilde\varphi(r)|\mid r\in R\}$ is finite, then $\varphi$ is $(G,X)$-controlled by a constant multiple of $f$. 
        Conversely, if (\ref{RelPre2}) is bounded and $\varphi$ is $(G,X)$-controlled, then $\sup\{|\tilde\varphi(r)|\mid r\in R\}$ is finite. 
    \end{prop}

    \begin{proof} 
        Let $B:=\sup\{|\tilde\varphi(r)|\mid r\in R\}$ and $D:=D(\varphi)$. For any $a\in \ker(\theta)$, 
        there exists $m\le f(|a|)$ and an expression \[a= \prod_{i=1}^{m} h_ir_ih_i^{-1},\] where $f_i\in K*F(X)$, $r_i\in R$. Note that $|\tilde\varphi(h_ir_ih_i^{-1})-\tilde\varphi(r_i)|\le 4D$ since $\tilde\varphi$ is a quasimorphism. Therefore, 
        \[|\tilde\varphi(a)|\le \sum_{i=1}^m |\tilde\varphi(r_i)|+5Dm \le (B+5D)f(|a|).\] 
        This shows that $\varphi$ is $(G,X)$-controlled by a constant multiple of $f$.

        Now we need to prove the second statement. By definition, for any $\varphi\in \hat Q(K)$ that is $(G,X)$-controlled by a function $f$,
        \[|\tilde\varphi(r)|\le f(|r|)\]
        for any $r\in R$. The right hand side is bounded because $R$ is bounded.
    \end{proof}

    \begin{cor}\label{LCTandCT}
        Assume that $G$ is weakly hyperbolic relative to $K$ and $X$. A quasimorphism $\varphi\in \hat Q(K)$ is linearly $(G,X)$-controlled if and only if $\varphi$ is $(G,X)$-controlled. 
    \end{cor}

    \begin{proof}
        By Lemma \ref{RelPresent}, there exists a bounded relative presentation $G=\langle K,X\mid R_K,R\rangle$ with a linear relative isoperimetric function. 
        Assume that $\varphi$ is $(G,X)$-controlled. By the second part of Proposition \ref{Prop_CTviaRP}, $\sup\{|\tilde\varphi(r)|\mid r\in R\}$ is finite. Therefore, $\varphi$ is linearly $(G,X)$-controlled by the first part of Proposition \ref{Prop_CTviaRP}.
    \end{proof}

\section{Extending controlled quasimorphisms}
\label{ProofMain}
    We are going to prove Theorem \ref{MainThm} in this section. Thoughout this section, let $G$ be a group and $K$ a subgroup. Let $X$ be a (symmetric) relative generating set of $G$ with respect to $K$. 
    Let $\varphi\in Q_{as}(K)$ with the defect $D(\varphi)=D$. Assume that $\varphi$ is linearly $(G,X)$-controlled with a constant $C_0>0$. For any $a\in K*F(X)$, we write $|a|_X$ to mean the number of $X$-letters (with multiplicity) in the normal form of $a$. 

    Now let $a\in \ker(\theta)$. If $|a|_X\ge 1$, since $\varphi$ is linearly controlled, we have
    \begin{align}
        |\tilde\varphi(a)|&\le C_0|a|+C_0 \nonumber\\
        &\le C_0(2|a|_X+1)+C_0 \nonumber\\
        &\le 4C_0|a|_X. \label{NewLC}
    \end{align}
    If $|a|_X=0$, then $a=1$, because $\theta$ is injective on $K$. By antisymmetricity, $\varphi(1)=0$, and thus (\ref{NewLC}) also holds. 
    We will use the new inequality (\ref{NewLC}) for linearly controlled quasimorphisms in the rest of this section. 
    
    Given a constant $C\ge 20C_0+11D$, define a function on $K*F(X)$ by
    \[\tilde\varphi_C(a):=\tilde\varphi(a)+C|a|_X.\]
    For any $g\in G$, define 
    \[\Phi_C(g):=\inf\{\tilde\varphi_C(a)\mid \theta(a)=g\}.\] 
    We will prove that the antisymmetrization of $\Phi_C$ is the quasimorphism we want.

\subsection{Auxiliary lemmata}
    The following coarse triangle inequalities mix those associated with $\tilde\varphi$ and $|\cdot|_X$ and are easy to use. 

    \begin{lem}\label{CoarseTriangle}
        For any $a,b\in K*F(X)$, we have
        \begin{enumerate}[label=\textnormal{(\roman*)}]
            \item $|\tilde\varphi_C(ab)-(\tilde\varphi_C(a)+\tilde\varphi(b))|\le C|b|_X+D$ 
            \item $|\tilde\varphi_C(ab)-(\tilde\varphi(a)+\tilde\varphi_C(b))|\le C|a|_X+D$ 
        \end{enumerate}      
    \end{lem} 

    \begin{proof}
        For (i), since $\tilde\varphi$ is a quasimorphism and $|\cdot|_X$ satisfies the triangle inequality, we estimate that 
        \begin{align*}
            |\tilde\varphi_C(ab)-(\tilde\varphi_C(a)+\tilde\varphi(b))|&=|\tilde\varphi(ab)+C|ab|_X-\tilde\varphi(a)-C|a|_X-\tilde\varphi(b)|\\
            &\le |\tilde\varphi(ab)-\tilde\varphi(a)-\tilde\varphi(b)|+C||ab|_X-|a|_X|\\
            &\le D+C|b|_X.
        \end{align*}
        Similarly, the inequality (ii) is also true.
    \end{proof}

    The following lemma shows that $\Phi_C$ is a well-defined real-valued function.
    
    \begin{lem}\label{MinExists}
        For any $g\in G$, $\Phi_C(g)$ is finite. Moreover, for any $k\in K$, $\Phi_C(k)=\varphi(k)$ and the only element in $\theta^{-1}(k)$ that realizes the infimum is $k$ itself. 
    \end{lem}
    
    \begin{proof}

        Fix any $g\in G$. Let $a, b\in \theta^{-1}(g)$. Obviously, $ab^{-1}\in \ker(\theta)$. By (\ref{NewLC}), $\tilde\varphi(ab^{-1})\ge -4C_0|ab^{-1}|_X$. Thus, 
        \[\tilde\varphi_C(ab^{-1})\ge (C-4C_0)|ab^{-1}|_X\ge 0.\]
        The second equality holds if and only if $a=b$. On the other hand, by Lemma \ref{CoarseTriangle}, $\tilde\varphi_C(ab^{-1})\le \tilde\varphi_C(a)-\tilde\varphi(b)+C|b|_X+D$. Therefore, 
        \[\tilde\varphi_C(a)\ge \tilde\varphi(b)-C|b|_X-D.\] 
        It follows that $\Phi_C(g)\ge \tilde\varphi(b)-C|b|_X-D$, which is finite. 

        For the ``moreover'' part, let $a\in \theta^{-1}(k)$ such that $a\ne k$. Note that $\tilde\varphi_C(k)=\varphi(k)$ and $|a|_X\ge 1$. By (\ref{NewLC}), 
        \begin{align*}
            \tilde\varphi_C(a)-\tilde\varphi_C(k)&=\tilde\varphi(a)+C|a|_X-\varphi(k)\\
            &\ge \tilde\varphi(ak^{-1})+C|a|_X-D\\
            &\ge -4C_0|ak^{-1}|_X+C|a|_X-D\\
            &= (C-4C_0)|a|_X-D\\
            &\ge C-4C_0-D\\
            &> 0,
        \end{align*}
        and hence the conclusion follows.
    \end{proof}

\subsection{The Cayley graph}
    Let $a\in K*F(X)$ and $\epsilon\ge 0$. If $\tilde\varphi_C(a)\le\Phi_C(\theta(a))+\epsilon$,  we say that $a$ is an $\epsilon$-\emph{minimal} element. Note that for any $\epsilon>0$ and any $g\in G$, there exists an $\epsilon$-minimal element in $\theta^{-1}(g)$. 
    Correspondingly, a path in $\Gamma(G,K\sqcup X)$ labeled by the normal form of an $\epsilon$-minimal element is called an $\epsilon$-\emph{minimal} path. 
    Clearly, the set of $\epsilon$-minimal paths are $G$-invariant. We fix any $\epsilon\in(0, C_0)$ in the following. 
    
    For any path $\gamma\subset\Gamma(G,K\sqcup X)$, we write $l(\gamma)$ to mean its length and use $\lab(\gamma)$ to denote its label. 
    For simplicity, we also write $\lab(\gamma)$ to mean the element in $K*F(X)$ represented by $\lab(\gamma)$ when $\lab(\gamma)$ is a normal form. 
    The following lemma shows that a subpath of an $\epsilon$-minimal path is $(4D+\epsilon)$-minimal. 
    
    \begin{lem}\label{Subpath}
        Let $\gamma\subset\Gamma(G,K\sqcup X)$ be an $\epsilon$-minimal path and $\alpha\subset\gamma$ a subpath. Then $\tilde\varphi_C(\lab(\alpha))\le \Phi_C(\theta(\lab(\alpha)))+4D+\epsilon$. 
    \end{lem}

    \begin{proof}        
        Assume that $\gamma=\beta_1\alpha\beta_2$. Let $a=\lab(\alpha)$, $b_1=\lab(\beta_1)$, $b_2=\lab(\beta_2)$. Let $a'\in \theta^{-1}(\theta(a))$ be an $\epsilon$-minimal element.  
        By assumptions, $\Delta:=\tilde\varphi_C(b_1a'b_2)-\tilde\varphi_C(b_1ab_2)\ge -\epsilon$. 
        
        On the other hand, note that $|b_1ab_2|_X=|b_1|_X+|a|_X+|b_2|_X$ since $\gamma$ is labeled by a normal form. We estimte that
        \begin{align*}
            \Delta&=\tilde\varphi(b_1a'b_2)-\tilde\varphi(b_1ab_2)+C|b_1a'b_2|_X-C|b_1ab_2|_X\\
            &\le \tilde\varphi(a')-\tilde\varphi(a)+4D+C|a'|_X-C|a|_X\\
            &=\tilde\varphi_C(a')-\tilde\varphi_C(a)+4D.
        \end{align*}
        In conclusion, $\tilde\varphi_C(a)\le \Phi_C(\theta(a))+ 4D+\epsilon.$
    \end{proof}

    \begin{lem}\label{Quasigeodesic}
        Any $\epsilon$-minimal path is a $(3,2)$-quasi-geodesic in $\Gamma(G,K\sqcup X)$. 
    \end{lem}

    \begin{proof}
        Let $\gamma\subset \Gamma(G,K\sqcup X)$ be an $\epsilon$-minimal path and $\alpha$ a subpath. Let $\beta$ be a path with the same endpoints as $\alpha$. Let $a=\lab(\alpha)$, $b=\lab(\beta)$. 
        
        By Lemma \ref{Subpath}, $\tilde\varphi_C(a)\le \Phi_C(\theta(a))+4D+\epsilon\le \tilde\varphi_C(b)+4D+\epsilon$. On the other hand, we estimate by (\ref{NewLC}) that
        \begin{align*}
            \tilde\varphi_C(b)-\tilde\varphi_C(a)&=\tilde\varphi(b)-\tilde\varphi(a)+C|b|_X-C|a|_X\\ &\le\tilde\varphi(ba^{-1})+D+C|b|_X-C|a|_X\\
            &\le 4C_0|ba^{-1}|_X+D+C|b|_X-C|a|_X\\
            &\le (C+4C_0)|b|_X-(C-4C_0)|a|_X+D
        \end{align*}
        Therefore, \[|a|_X\le \frac{C+4C_0}{C-4C_0}|b|_X+\frac{5D+\epsilon}{C-4C_0}.\]
        
        Note that $l(\alpha)=|a|$ because $\lab(\gamma)$ is a normal form. Thus, $l(\alpha)\le 2|a|_X+1$. On the other hand, $l(\beta)\ge |b|\ge |b|_X$. 
        Choosing $\beta$ as a geodesic, we conclude that $\gamma$ is a $(\lambda,\mu)$-quasi-geodesic, where $\lambda=\frac{2C+8C_0}{C-4C_0}\le 3$ and $\mu=\frac{10D+2\epsilon}{C-4C_0}+1\le 2$.
    \end{proof}

    \begin{lem}\label{WeakQM}
        For any $\delta,m\ge 0$, there exists $M\ge 1$ such that the following holds. Assume that $\Gamma(G,K\sqcup X)$ is $\delta$-hyperbolic.
        For any $g,h\in G$, if $g$ lies within distance $m$ from a geodesic from $1$ to $gh$, then $|\Phi_C(g)+ \Phi_C(h)-\Phi_C(gh)|\le MC$.
    \end{lem}

    \begin{proof}
        Let $\alpha_1$ be an $\epsilon$-minimal path from $1$ to $g$, let $\alpha_2$ be an $\epsilon$-minimal path from $g$ to $gh$, and let $\beta$ be an $\epsilon$-minimal path from $1$ to $gh$. By projecting $g$ to a closest point $g'$ on $\beta$, we divide $\beta$ into $\beta_1\beta_2$. Let $\gamma$ be an $\epsilon$-minimal path from $g$ to $g'$. Let $a_1,a_2,b,b_1,b_2,c$ be the labels of $\alpha_1,\alpha_2,\beta,\beta_1,\beta_2,\gamma$, respectively. By definition, $\tilde\varphi_C(a_1)\le \Phi_C(g)+\epsilon$, $\tilde\varphi_C(a_2)\le \Phi_C(h)+\epsilon$ and $\tilde\varphi_C(b)\le \Phi_C(gh)+\epsilon$. 
        
        By Lemma \ref{Subpath} and Lemma \ref{CoarseTriangle}, we estimate for the triangle $\alpha_1\gamma\beta_1^{-1}$ that
        \begin{align*}
            \tilde\varphi_C(b_1)&\le \Phi_C(\theta(b_1))+4D+\epsilon\\
            &\le \tilde\varphi_C(a_1c)+4D+\epsilon\\
            &\le \tilde\varphi_C(a_1)+\tilde\varphi(c)+C|c|_X+5D+\epsilon.
        \end{align*}
        Similarly, we also have 
        \[\tilde\varphi_C(a_1)\le \tilde\varphi_C(b_1)-\tilde\varphi(c)+C|c|_X+5D+\epsilon.\]
        Therefore, 
        \begin{align}\label{FirstTriangle}
            |\tilde\varphi_C(a_1)+\tilde\varphi(c)-\tilde\varphi_C(b_1)|\le C|c|_X+5D+\epsilon.
        \end{align}

        For the triangle $\alpha_2\beta_2^{-1}\gamma^{-1}$, we can estimate in a similar way and obtain
        \begin{align}\label{SecondTriangle}
        |\tilde\varphi_C(a_2)-\tilde\varphi_C(b_2)-\tilde\varphi(c)|\le C|c|_X+5D+\epsilon.
        \end{align}

        Combining (\ref{FirstTriangle}) with (\ref{SecondTriangle}) and note that $|\tilde\varphi_C(b_1)+\tilde\varphi_C(b_2)-\tilde\varphi_C(b)|\le D$ and $C\ge 11D+2\epsilon$, we obtain
        \[|\tilde\varphi_C(a_1)+\tilde\varphi_C(a_2)-\tilde\varphi_C(b)|\le (2|c|_X+1)C.\]
        Since $\gamma$ is a $(3,2)$-quasi-geodesic, we have $|c|_X\le l(\gamma)\le 3d(g,g')+2$. Since $\beta$ is a $(3,2)$-quasi-geodesic, the distance $d(g,g')=d(g,\beta)$ is bounded above by a constant depending only on $\delta,m$ by Morse Lemma. Thus, the conclusion follows. 
    \end{proof}

    \begin{proof}[Proof of Theorem \ref{MainThm}]
        By Lemma \ref{LCTandCT}, we can assume that $\varphi$ is linearly $(G,X)$-controlled with a constant $C_0$. Therefore, we can follow the discussions in this section. 
        For any $g,h\in G$, consider a geodesic triangle with vertices $1,g,gh$. Since $\Gamma(G,K\sqcup X)$ is $\delta$-hyperbolic, there exists $c\in G$ such that $c$ lies within distance $\delta$ from each side. Let 
        \[\Phi'_C(g):=\frac{1}{2}(\Phi_C(g)-\Phi_C(g^{-1}))\]
        be the antisymmetrization of $\Phi_C$. Clearly, $\Phi'_C$ inherits the property described in Lemma \ref{WeakQM} from $\Phi_C$. Therefore, there exists $M\ge 0$ that only depends on $\delta$ such that 
        \begin{align*}
            |\Phi'_C(c)+\Phi'_C(c^{-1}g)-\Phi'_C(g)|&\le MC,\\
            |\Phi'_C(g^{-1}c)+\Phi'_C(c^{-1}gh)-\Phi'_C(h)|&\le MC,\\
            |\Phi'_C(c)+\Phi'_C(c^{-1}gh)-\Phi'_C(gh)|&\le MC.
        \end{align*}

        Combining these inequalities and noting that $\Phi'_C(c^{-1}g)+\Phi'_C(g^{-1}c)=0$ by antisymmetricity, we have
        \[|\Phi'_C(g)+ \Phi'_C(h)-\Phi'_C(gh)|\le 3MC+D.\]
        This proves that $\Phi'_C$ is a quasimorphism with the defect at most $3MC+D$. By choosing $C=20C_0+11D$, we see that $D(\Phi'_C)\le 60MC_0+34MD$. By Lemma \ref{MinExists}, $\Phi'_C|_K=\varphi$. In conclusion, $\varphi$ extends to $\Phi'_C\in Q_{as}(G)$.
    \end{proof}

\section{Normal subgroups}
\label{Sec_Normal}

\subsection{A criterion for being controlled}

    In this subsection, we are going to give sufficient conditions for quasimorphisms to be controlled, including a general criterion (Proposition \ref{CTNormal}) and its corollaries. Throughout this subsection, let $G$ be a group and $K$ a normal subgroup. Let $\bar G=G/K$.
    Consider the short exact sequence
    \begin{align}\label{SES1}
        1\to K\to G\xrightarrow{\pi} \bar G\to 1.
    \end{align}
    
    Assume that there are presentations $K=\langle Y\mid S\rangle$ and $\bar G=\langle \bar X\mid \bar R\rangle$. For each $\bar x\in\bar X$, choose any $x\in \pi^{-1}(\bar x)$, and let $X:=\{x\mid \bar x\in \bar X\}$. This induces a natural map $s:\bar X^*\to X^*$ called a \emph{lift} of $\bar X^*$. The quotient map $\pi$ induces a natural map from $(X\sqcup Y)^*$ to $\bar X^*$. For simplicity, we still use $\pi$ to denote this map. 

    For any $\bar r\in \bar R$, its lift $r:=s(\bar r)$ represents an element in $K$. Let $k_r\in K$ be the element represented by $r$. Let $w_r\in Y^*$ be a word that represents $k_r$, and let 
    \[R:=\{r w_r^{-1}\mid \bar r\in\bar R\}.\] 
    For each $x\in X$ and $y\in Y$, the word $xyx^{-1}$ represents an element in $K$, and thus is represented by a word $w_{xy}\in Y^*$. Let
    \[T:=\{xyx^{-1}w_{xy}^{-1}\mid x\in X, y\in Y\}.\]
    The following lemma seems to be well-known (see, for example \cite[Proposition 2.55]{HEO05} or \cite[Lemma 2.1]{BW11}).
    
    \begin{lem}\label{PreExtension}
        With the above notations, $\langle X, Y\mid R, S, T\rangle$ is a presentation of $G$. 
    \end{lem}

    For our interest, the presentation of $K$ will be fixed as $\langle K\mid R_K\rangle$.     
    In this case, each $w_r$ or $w_{xy}$ can be chosen as a unique $K$-letter, and the presentation provided by Lemma \ref{PreExtension} can be seen as a relative presentation with respect to $K$. We rewrite $R$ as 
    \[R=\{rk_r^{-1}\mid \bar r\in\bar R\}.\]
    For each $w\in X^*$ and $k\in K$, the word $wkw^{-1}$ represents an element $k_w\in K$. Thus, we can expand $T$ to 
    \[T':=\{wkw^{-1}k_w^{-1}\mid w\in X^*, k\in K\}. \]

    The following lemma is a variation of the Bryant Park lemma \cite{BBM20}.
    
    \begin{lem}\label{LinearRIF}
        If the presentation $\bar G=\langle \bar X\mid \bar R\rangle$ has an isoperimetric function $f(n)$, then the relative presentation $G=\langle K, X\mid R_K, R, T'\rangle$ has a relative isoperimetric function $2f(n)+n$. 
    \end{lem}

    \begin{proof}
        Consider any $a\in \ker(\theta)$ of word length $n$. Without loss of generality, assume that the normal form $w_a$ of $a$ starts with a $K$-letter and ends with an $X$-letter. There is an expression 
        \[w_a=\prod_{i=1}^mk_iw_i,\]
        where $w_i\in X^*$. We need to reduce $w_a$ to the empty word in the group $G$ using relations from $R_K\cup R\cup T'$ and bound the number of times a relation from $R\cup T'$ is used. 

        It is easy to see that 
        \[w_a=_{K*F(X)}\prod_{i=1}^m w_i\cdot \prod_{i=1}^mv_i,\]
        where $v_i=(\prod_{j=i}^mw_j)^{-1}\cdot k_i\cdot \prod_{j=i}^mw_j$. 
        Thus, we can use at most $n$ relations from $T'$ to reduce $w_a$ to $w_a'=\prod_{i=1}^mw_i\in X^*$. 
        
        The word $w_a'$ projects to $\bar w_a'\in \bar X^*$ under the map $\pi$. Since $w_a$ represents the identity of $G$, $\bar w_a'$ represents the identity of $\bar G$. Therefore, there is an expression 
        \[\bar w_a'=_{F(\bar X)}\prod_{j=1}^l\bar f_j\bar r_j\bar f_j^{-1},\]
        where $\bar r_j\in R$, $\bar f_j\in X^*$, and $l\le f(n)$. By lifting this equality, we have 
        \begin{align*}
            w_a'=_{F(X)} \prod_{j=1}^lf_j r_j f_j^{-1}.
        \end{align*}
        Now we can use at most $f(n)$ relations from $R$ and at most $f(n)$ relations from $T'$ to reduce each subword $f_j r_j f_j^{-1}$ to a single $K$-letter. Till now, $w_a$ has been reduced into a word on $K$ that represents the identity. We are done since it is a relation from $R_K$. 

        During the whole process, we use at most $f(n)+n$ relations from $T'$ and $f(n)$ relations from $R$ so the conclusion follows.
    \end{proof}

    The following criterion is very useful for extending quasimorphisms from normal subgroups.

    \begin{prop}\label{CTNormal}
        Let $\varphi\in \hat Q(K)^G$. 
        Assume that the presentation $\bar G=\langle \bar X\mid \bar R\rangle$ has a linear isoperimetric function. If $\sup\{|\varphi(k_r)|\mid r\in R\}$ is finite, then $\varphi$ is linearly $(G,X)$-controlled. 
    \end{prop}

    \begin{proof}
        By Lemma \ref{LinearRIF}, the relative presentation $G=\langle K, X\mid R_K, R, T'\rangle$ has a linear relative isoperimetric function. By Proposition \ref{Prop_CTviaRP}, we only need to prove that $\sup\{|\tilde\varphi(v)|\mid v\in R\cup T'\}<\infty$. 

        First, note that $|\tilde \varphi(rk_r^{-1})|=|\varphi(k_r^{-1})|$ for any $rk_r^{-1}\in R$. By assumption, this gives $\sup\{|\tilde\varphi(v)|\mid v\in R\}<\infty$. Second, for any $wkw^{-1}k_w^{-1}\in T'$, $k$ and $k_w$ are conjugate in $G$. Since $\varphi$ is $G$-quasi-invariant, $|\varphi(k)-\varphi(k_w)|\le D^G(\varphi)$ for any $wkw^{-1}k_w^{-1}\in T'$. Thus, 
        \begin{align*}
            |\tilde \varphi(wkw^{-1}k_w^{-1})|&=|\varphi(kk_w^{-1})|\\
            &\le |\varphi(k)-\varphi(k_w)|+3D(\varphi)\\
            &\le D^G(\varphi)+3D(\varphi).
        \end{align*} This shows that $\sup\{|\tilde\varphi(v)|\mid v\in T'\}<\infty$. 
        
        In conclusion, $\varphi$ is linearly $(G,X)$-controlled. 
    \end{proof}

    A  section $s:\bar G\to G$ is called a \textit{quasi-splitting} if the set 
    \[\Delta(s):=\{s(\bar g_2)^{-1}s(\bar g_1)^{-1}s(\bar g_1\bar g_2)\mid \bar g_1,\bar g_2\in \bar G\}\]
    is finite. Following \cite{FK16}, we say that (\ref{SES1}) quasi-splits if it admits a quasi-splitting. For example, every  section $s:\bar G\to G$ is a quasi-splitting if $\bar G$ is finite.

    \begin{cor}\label{Qsplit}
        If (\ref{SES1}) admits a quasi-splitting $s:\bar G\to G$, then any $\varphi\in \hat Q(K)^G$ is linearly $(G,s(\bar G))$-controlled. 
    \end{cor}

    \begin{proof}
        Consider the trivial presentation $\bar G=\langle\bar G\mid R'_{\bar G}\rangle$, where $R'_G\subset R_G$ consists of the words of length at most $3$. Clearly, this presentation has a linear isoperimetric function. The quasi-splitting $s$ induces a lift map $\tilde s:\bar G^*\to G^*$. Let $\bar r=\bar g_1\bar g_2\bar g_3\in R'_{\bar G}$. Since $\bar g_3=_{\bar G}\bar g_2^{-1}\bar g_1^{-1}$, we have 
        \begin{align*}
            k_{\tilde s(\bar r)}&=\tilde s(\bar g_1\bar g_2\bar g_3)\\
            &=_G s(\bar g_1)s(\bar g_2)s(\bar g_2^{-1}\bar g_1^{-1})\\
            &\subset_G s(\bar g_1)s(\bar g_1^{-1})\Delta(s)^{-1}\\
            &\subset_G s(1)\Delta(s)^{-2}.
        \end{align*}
        Since $\Delta(s)$ is finite, Proposition \ref{CTNormal} tells us that any $\varphi\in Q_{as}(K)^G$ is linearly $(G,s(\bar G))$-controlled. 
    \end{proof}

    A \emph{transversal} of a normal subgroup $K$ in $G$ is defined to be a set of coset representatives of $K$ in $G$. 

    \begin{cor}\label{RHGQuo}
        Let $G$ be a group and $K$ a normal subgroup. Assume that there exists $H<G$ such that the quotient group $G/K$ is hyperbolic relative to $H/(K\cap H)$. Let $\varphi\in Q_{as}(K)^G$. If there exists a transversal $L$ of $K\cap H$ in $H$ such that the restricted quasimorphism $\varphi|_{K\cap H}$ is $(H,L)$-controlled, then $\varphi$ extends to an antisymmetric quasimorphism on $G$. 
    \end{cor}

    \begin{proof}
        Let $\bar G:=G/K$ and $\bar H:=H/(K\cap H)$. By Definition \ref{DefRHG}, $\bar G$ has a presentation 
        \[\bar G=\langle \bar H,\bar X\mid R_{\bar H}',\bar R\rangle\]
        with a linear isoperimetric function, where $\bar R$ is finite. 

        Let $\pi:G\to \bar G$ be the quotient map, and let $L$ be the given transversal of $K\cap H$ in $H$. As before, we also use $\pi$ to denote the map from $(H\sqcup X)^*$ or $(K\sqcup L\sqcup X)^*$ to $(\bar H\sqcup\bar X)^*$ induced by $\pi$. Let $s:(\bar H\sqcup\bar X)^*\to (L\sqcup X)^*$ be a lift map. By Proposition \ref{CTNormal} and Theorem \ref{MainThm}, we only need to prove that $\sup\{|\varphi(k_{s(\bar r)})|\mid \bar r\in R_{\bar H}'\cup \bar R\}<\infty$, where $k_{s(\bar r)}\in K$ is the element represented by the word $s(\bar r)$. 

        Let $\bar r\in R_{\bar H}'$. Suppose that $s(\bar r)$ has the form $s(\bar r)=l_1l_2l_3$, where $l_i\in L$. Note that $k_{s(\bar r)}\in K\cap H$, and thus $l_1l_2l_3k_{s(\bar r)}^{-1}\in (L\sqcup (K\cap H))^*$. Since $\varphi|_{K\cap H}$ is $(H,L)$-controlled by a function $f$, we have $|\varphi(k_{s(\bar r)})|\le f(4)$. Since $\bar R$ is finite, we know that $\sup\{|\varphi(k_{s(\bar r)})|\mid \bar r\in R_{\bar H}'\cup \bar R\}<\infty$.
    \end{proof}

    Clearly, Corollary \ref{HypQuo} follows from Corollary \ref{RHGQuo}.

\subsection{Group-theoretic Dehn filling}
    
    Let $H\hookrightarrow_h (G,X)$. By Definition \ref{DefHypEmb}, there exists a strongly bounded relative presentation 
    \[G=\langle H,X\mid R_H,R\rangle\]
    with a linear relative isoperimetric function.
    Given a subgroup $N\lhd H$, the quotient group $\bar G:=G/\llangle N \rrangle^G$ is called the \emph{Dehn filling} of $G$ with respect to $N$.
    
    Let $K:=\llangle N \rrangle^G$, the normal closure of $N$ in $G$. Consider the short exact sequence
    \[1\to K\to G\xrightarrow{\pi} \bar G\to 1.\]
    Let $\bar H:=\pi(H)\cong H/(K\cap H)$ and $\bar X:=\pi(X)$. We still use $\pi$ to denote the map from $(H\sqcup X)^*$ to $(\bar H\sqcup\bar X)^*$ induced by $\pi$, where $\bar H:=\pi(H)\cong H/(K\cap H)$ and $\bar X:=\pi(X)$.

    The following is part of \cite[Theorem 7.19]{DGO17}.
    \begin{thm}\label{DGOFill}
        With the above notations, there exists a finite subset $F\subset H-\{1\}$ such that for any $N\lhd H$ satisfying that $N\cap F=\emptyset$, the following hold.
        \begin{enumerate}
            \item The natural map from $H/N$ to $\bar G$ is injective. In other words, $K\cap H=N$. 
            \item $\bar G$ has a strongly bounded relative presentation 
            $\bar G=\langle \bar H,\bar X\mid R_{\bar H},\pi(R)\rangle$
            with a linear relative isoperimetric function. 
            \item The restricted map $\pi|_X: X\to \bar X$ is bijective. 
        \end{enumerate}
    \end{thm}

    Now we can state the main result of this subsection. It shows that the extendability of a quasimorphism $\varphi\in Q_{as}(K)^G$ only depends on its restriction on $N$. 
    
    \begin{thm}\label{QMDGOFill}
        Let $H\hookrightarrow_h G$, and let $F$ be the finite subset provided by Theorem \ref{DGOFill}. Let $N\lhd H$ be a subgroup such that $N\cap F=\emptyset$. Let $K=\llangle N \rrangle^G$ and $\varphi\in Q_{as}(K)^G$. If there exists a transversal $L$ of $N$ in $H$ such that the restricted quasimorphism $\varphi|_N$ is $(H,L)$-controlled, then $\varphi$ extends to $G$. 
    \end{thm}
    
    \begin{proof}
        Let $X\subset G$ be a relative generating set such that $H\hookrightarrow_h (G,X)$. Let $s:(\bar H\sqcup\bar X)^*\to (L\sqcup X)^*$ be the lift map induced by the transversal $L$ and the map $(\pi|_X)^{-1}$. By Theorem \ref{DGOFill}(2) and Lemma \ref{RelPresent}, $\bar G$ has a strongly bounded presentation 
        \[\bar G=\langle \bar H,\bar X\mid R_{\bar H}',\pi(R)\rangle\]
        with a linear isoperimetric function. 
        By Proposition \ref{CTNormal} and Theorem \ref{MainThm}, we only need to prove that $\sup\{|\varphi(k_{s(\bar r)})|\mid \bar r\in R_{\bar H}'\cup \pi(R)\}<\infty$, where $k_{s(\bar r)}\in K$ is the element represented by the word $s(\bar r)$. 
        
        \textit{Case 1.} Let $\bar r\in R_{\bar H}'$. Suppose that $s(\bar r)$ has the form $s(\bar r)=l_1l_2l_3$, where $l_i\in L$. Note that $k_{s(\bar r)}\in K\cap H=N$, and thus $l_1l_2l_3k_{s(\bar r)}^{-1}\in (L\sqcup N)^*$. Since $\varphi|_N$ is $(H,L)$-controlled by a function $f$, we have $|\varphi(k_{s(\bar r)})|\le f(4)$. 
        
        \textit{Case 2.} Let $\bar r=\pi(r)$, where $r\in R$. Suppose that $r$ has the form
        \[r=\prod_{i=1}^m w_ih_i,\]
        where $w_i\in X^*$ and $h_i\in H$. Each $h_i$ can be represented by a unique word $l_in_i$, where $l_i\in L$ and $n_i\in N$. Specifically, $l_i=s(\pi(h_i))$ and $n_i=_Hs(\pi(h_i))^{-1}h_i$. By definition, 
        \[s(\bar r)=\prod_{i=1}^m w_il_i.\]
        It is easy to see that 
        \[r=_{K*F(L\sqcup X)}\prod_{i=1}^m w_il_i\cdot \prod_{i=1}^mv_i,\]
        where $v_i=(\prod_{j=i+1}^mw_jl_j)^{-1}\cdot n_i\cdot \prod_{j=i+1}^mw_jl_j$. Therefore, 
        \[k_{s(\bar r)}=_G(\prod_{i=1}^m v_i)^{-1}.\]
        Let $k_i\in K$ be the element represented by $v_i$.  
        Since $\varphi$ is $G$-quasi-invariant, $|\varphi(k_i)-\varphi(n_i)|\le D^G(\varphi)$. Therefore, 
        \begin{align*}
            |\varphi(k_{s(\bar r)})| &\le \sum_{i=1}^m|\varphi(k_i)|+(m-1)D(\varphi)\\
            &\le \sum_{i=1}^m|\varphi(n_i)|+mD^G(\varphi)+(m-1)D(\varphi).
        \end{align*}
        Since $R$ is bounded, the number $m$ is uniformly bounded. Since there are only finitely many $h\in H$ that appear in words from $R$, the set 
        \[\{n\in N\mid n=s(\pi(h))^{-1}h, \text{ where } h\in H \text{ appear in words from $R$}\}\]
        is finite. Thus, $\sup\{|\varphi(k_{s(\bar r)})|\mid \bar r\in \pi(R)\}$ is finite. 
    \end{proof}

\subsection{Mapping class groups}
    A \emph{finite-type} surface is a closed surface with finitely many punctures. Let $\Sigma$ be a finite-type surface, possibly with boundary $\partial\Sigma$. 
    The mapping class group $\mcg(\Sigma)$ is the group of homeomorphisms of $\Sigma$ restricting to the identity on $\partial\Sigma$, modulo isotopy. If there are punctures $P=\{p_1,\dots,p_n\}$ on $\Sigma$, we define $\mcg(\Sigma,P)$ to be the finite-index subgroup of $\mcg(\Sigma)$ that fixes each puncture. 
    
    \begin{thm}\label{FillpA}
        Let $\Sigma$ be a finite-type surface, and let $G=\mcg(\Sigma)$. Let $g$ be a pseudo-Anosov mapping class. There exists $N\ge 1$ such that the following holds. For any $k\in\N$, any antisymmetric $G$-quasi-invariant quasimorphism on $\llangle g^{kN}\rrangle^G$ extends to $G$. 
    \end{thm}

    \begin{proof}
        As shown in \cite{BF02}, the given $g$ acts as a WPD element on the curve graph of $\Sigma$ (see also \cite{PS17} for a simpler proof). This implies that the cyclic subgroup $\langle g\rangle$ is contained in a unique maximal elementary subgroup $H$ by \cite[Lemma 6.5]{DGO17}. Moreover, $H$ is hyperbolically embedded in $G$ by \cite[Theorem 6.8]{DGO17}. 
        
        Let $F\subset H-\{1\}$ be the finite subset provided by Theorem \ref{DGOFill}. We can choose $N$ sufficiently large such that $\langle g^N\rangle\lhd H$ and $\langle g^N\rangle\cap F=\emptyset$. Let $\varphi\in Q_{as}(K)^G$, where $K:=\llangle g^{kN}\rrangle^G$. Since $\langle g^N\rangle$ is a normal subgroup of $H$ of finite index, Corollary \ref{Qsplit} tells us that the restricted quasimorphism $\varphi|_{\langle g^N\rangle}$ is linearly $(H,L)$-controlled for any transversal $L$. 
        Finally, we can apply Theorem \ref{QMDGOFill} and conclude. 
    \end{proof}
    
    Let $\Sigma$ be a finite-type surface with nonempty boundary components $\gamma_1,\dots,\gamma_n$. 
    Let $\hat\Sigma$ be the surface obtained by gluing a once-punctured disk with puncture $p_i$ to each $\gamma_i$, and let $P=\{p_1,\dots,p_n\}$. 
    There is a central extension called the \emph{capping sequence}
    \begin{align*}
        1\to T_{\Gamma} \to \mcg(\Sigma)\xrightarrow{\pi} \mcg(\hat\Sigma,P)\to 1,
    \end{align*}
    where $T_{\Gamma}$ is the free abelian subgroup generated by Dehn twists around every $\gamma_i$ (see \cite[\S 3.6]{FM12}). 
    By \cite[Propositon 6.4]{FK16}, the capping sequence has a bounded Euler class. Let $g$ be a pseudo-Anosov mapping class of $\Sigma$, and let $K:=\llangle g^{kN}\rrangle^G$, where $N\in\N$ is provided by Theorem \ref{FillpA} and $k\in\N$ is arbitrary. By \cite[Proposition 4.4]{FMS25}, $K\cap  T_{\Gamma}=\emptyset$. Thus, we obtain a quotient central extension
    \begin{align}\label{QuoCent}
            1\to T_{\Gamma}\to \mcg(\Sigma)/K\to \mcg(\hat\Sigma,P)/\pi(K)\to 1.
    \end{align}

    \cite[Proposition 2.16]{FMS25} says that if a certain (antisymmetric) quasi-invariant quasimorphism on $\pi(K)$ extends to $\mcg(\hat\Sigma,P)$, then the quotient central extension (\ref{QuoCent}) has a bounded Euler class. Theorem \ref{FillpA} shows that it is indeed the case here. 
    
    \begin{cor}
        The quotient central extension (\ref{QuoCent}) has a bounded Euler class. 
    \end{cor}

    The base group $\mcg(\hat\Sigma,P)/\pi(K)$ is known to be hierarchically hyperbolic by \cite[Theorem 6.2]{BHS17a}. Using \cite[Theorem 5.14]{AHPZ23}, we have shown that

    \begin{cor}\label{QuoMCG}
        $\mcg(\Sigma)/K$ is hierarchically hyperbolic. 
    \end{cor}

    This answers \cite[Question 4.2]{FMS25} affirmatively. 

\subsection{Mixed stable commutator length}

Let $G$ be a group and $K$ a normal subgroup. In their paper \cite{KKMM22}, Kawasaki, Kimura, Matsushita, and Mimura introduce the (stable) mixed commutator length on the subgroup $[G,K]$, generalizing the classical (stable) commutator length defined on commutator subgroups. We briefly review their definitions. 

A $(G,K)$-\emph{commutator} is an element in $[G,K]$ of the form $[g,k]=gkg^{-1}k^{-1}$ for some $g\in G$ and $k\in K$. Let $[G,K]$ be the subgroup generated by $(G,K)$-commutators. 
The $(G,K)$-\emph{commutator length} $\cl_{G,K}(g)$ of an element $g\in [G,K]$ is defined to be the word length of $g$ with respect to the generating set consisting of all $(G,K)$-commutators. It is known that the limit
\[\scl_{G,K}(g)=\lim_{n\to \infty}\frac{\cl_{G,K}(g^n)}{n}\]
exists, and we call $\scl_{G,K}(g)$ the stable $(G,K)$-commutator length of $g$. When $K$ equals $G$, $\cl_{G,G}$ and $\scl_{G,G}$ equal the commutator length and the stable commutator length, respectively. In this case, we write $\cl_G$ and $\scl_G$ instead.

The celebrated Bavard's duality theorem \cite{Bav91} (see also \cite{Cal09a}) relates stable commutator length to homogeneous quasimorphisms. Recently, this duality is generalized to mixed stable commutator length and invariant quasimorphisms in \cite{KKMM22}. In particular, for any $g\in [G,K]$, $\scl_{G,K}(g)>0$ if and only if there is a $G$-invariant quasimorphism $\varphi$ on $K$ with $\varphi(g)\ne 0$. Here we say that a quasimorphism $\varphi$ on $K$ is $G$-\emph{invariant} if $\varphi(gkg^{-1})=\varphi(k)$ for any $g\in G$ and $k\in K$. Clearly, every $G$-invariant quasimorphism on $K$ is antisymmetric and $G$-quasi-invariant. 
From this duality, Kawasaki, Kimura, Maruyama, Matsushita, and Mimura \cite{KKMMM25} prove the following as part of their Theorem 2.1.

\begin{thm}\label{biLip}
    If every $G$-invariant quasimorphism on $K$ extends to $G$, then $\scl_G$ and $\scl_{G,K}$ are bi-Lipschitz equivalent on $[G,K]$.
\end{thm}

Combining Theorem \ref{biLip} with Corollary \ref{HypQuo} and Theorem \ref{FillpA} respectively, we have

\begin{cor}
~
    \begin{enumerate}
        \item If $G/K$ is hyperbolic, then $\scl_G$ and $\scl_{G,K}$ are bi-Lipschitz equivalent on $[G,K]$.
        \item If $G=\mcg(\Sigma)$ and $K=\llangle g^{kN}\rrangle^G$ as in Theorem \ref{FillpA}, then $\scl_G$ and $\scl_{G,K}$ are bi-Lipschitz equivalent on $[G,K]$. 
    \end{enumerate}
\end{cor}

\printbibliography
\end{document}